\documentclass[a4paper, 12pt]{article}
\usepackage[english]{babel}

\usepackage{amssymb}
\usepackage{amsthm}
\usepackage{amsmath}
\usepackage{amsfonts}
\usepackage{mathrsfs}

\newtheorem{theorem}{Theorem}
\newtheorem{definition}{Definition}
\newtheorem{corollary}{Corollary}
\newtheorem{example}{Example}
\newtheorem{lemma}{Lemma}
\newtheorem{proposition}{Proposition}

\newcommand{\comment}[1]{}
\numberwithin{equation}{section}

\begin{document}

\title{Banff designs: difference methods for coloring incidence graphs}{}

\author{Marco Buratti\thanks{SBAI, Sapienza University of Rome, Rome, Italy, {\tt marco.buratti@uniroma1.it}}, 
Francesca Merola\thanks{Dipartimento di Matematica e Fisica, Universit\`a Roma Tre, Rome, Italy
{\tt merola@mat.uniroma3.it }}, 
Anamari Naki\'c\thanks{Faculty of Electrical Engineering and Computing, University of Zagreb, Croatia, {\tt anamari.nakic@fer.hr}}, \\
Christian Rubio-Montiel \thanks{Divisi{\' o}n de Matem{\' a}ticas e Ingenier{\' i}a, FES Acatl{\' a}n, Universidad Nacional Aut{\'o}noma de M{\' e}xico, Naucalpan, Mexico, {\tt christian.rubio@acatlan.unam.mx}}}


\maketitle

\begin{abstract}
We present some results on the harmonious colorings of the Levi graph of a  $2$-design, focusing on Steiner $2$-designs. It is easily seen that the harmonious chromatic number of such a Levi graph is at least the number of points of the design:  we study and construct \emph{Banff designs}, that is, designs such that this lower bound is attained. 
\end{abstract} 

\section{Introduction}


A \emph{harmonious coloring} of a finite graph $G$ is a proper $k$-coloring of its vertices such that every pair of colors appears on at most one edge.  The \emph{harmonious chromatic number} $h(G)$ of $G$ is the minimum number $k$ such that $G$ has a harmonious $k$-coloring. 

This parameter was introduced by Mitchem \cite{MR989131}, a slight variation of the original definition given independently by Hopcroft and Krishnamoorthy \cite{MR711339} and by Frank, Harary, and Plantholt \cite{MR683991}.

From the definition, a graph $G$ of order $n$ has size $m \leq \binom{h(G)}{2}$, and this gives the lower bound
\begin{equation}\label{eq1}
\sqrt{2m+\frac{1}{4}}+\frac{1}{2} \leq h(G). 
\end{equation}

If $G$ has order $n$ and diameter at most two, then $h(G)=n$, see \cite{MR1477743}. In the survey \cite{MR1477743}, we can find a list of interesting results about the harmonious chromatic number and the achromatic number, a related parameter. 

In this note, inspired by the recent preprint \cite{Christian}, we present some results concerning the harmonious chromatic number of the Levi graph of a $(v,k,\lambda)$-design. It is readily seen that this chromatic number is at least $v$; we will be interested in studying designs for which this number is exactly $v$.

We call designs whose Levi graph can be colored with $v$ colors {\it Banff Designs} for brevity. Indeed, the work leading to this note started at the BIRS workshop ``Extremal Graphs arising from Designs and Configurations'', held in Banff in May 2023. Results on the harmonious chromatic number of the Levi graph of the complete graph \cite{GMetal} were also part of the output of the workshop.

\section{The harmonious chromatic number of Levi graphs of incidence structures}

A linear space $\mathcal{S}$ is an incidence point-line structure $(V,\mathscr L)$ such that
any two distinct points are on exactly one line, any line has at least two points, and $\mathscr L$ has at least two lines.

The incidence graph or {\it Levi graph} $G$ of a linear space $\mathcal{S}=(V,\mathscr L)$ is the bipartite graph with $|V|+|\mathscr L|$ vertices corresponding 
to the points and the lines of $\mathcal{S}$, where two vertices are adjacent if the corresponding point-line pair is incident.

On one hand, in any harmonious coloring of $G$, any two different vertices in $V$ receive different colors, because the corresponding two points share a line, so that  $h(G)\geq |V|=v$.

On the other hand, a \emph{line coloring} of $\mathcal{S}$ with $k$ colors is an assignment of the lines of $\mathcal{S}$ to a set of $k$ colors. A line coloring of $\mathcal{S}$ is called \emph{proper} if any two intersecting lines have different colors. The \emph{chromatic index} $\chi'(\mathcal{S})$ of $\mathcal{S}$ is the smallest $k$ such that there exists a proper line coloring of $\mathcal{S}$ with $k$ colors; then we have $v\leq h(G) \leq v+\chi'(\mathcal{S})$. Erd{\H o}s, Faber and Lov{\' a}sz conjectured that the chromatic index of any finite linear space $\mathcal{S}$ cannot exceed the number of its points (see \cite{AKRV,Erd1,Erd2}), i.e., $\chi'(\mathcal{S})\leq v.$ Therefore, if the conjecture holds, we have $v\leq h(G) \leq 2v$, see \cite{Christian}.



As is well known, a 2-$(v,k,\lambda)$-{\it design}, or simply $(v,k,\lambda)$-{\it design} $\cal D$, is a pair $(V,{\cal B})$, where $V$ is a $v$-set of points and 
$\cal B$ is a set of $k$-subsets of $V$ called blocks, $v\ge k \ge 3$, having the property that there are precisely $\lambda$ blocks containing each pair of points of $V$ 
(see for instance \cite{BJL, Stin}). It is easily seen that each point lies on a constant number of blocks denoted by $r$ (the {\it replication number}), while the number of blocks is denoted by $b$. These parameters are connected by the two identities $vr=bk$ and $r(k-1)= \lambda(v-1)$.  
A $(v,k,\lambda)$ design with $\lambda=1$ is also called a {\it Steiner $2$-design}.

We may consider the Levi graph of any incidence structure and in this paper we focus on the Levi graph $\cal L$ of a design $\cal D$. As above, we have that the harmonious chromatic number is bounded below by the number of points, $h({\cal L})\ge v$. We will be concerned with studying and constructing designs, especially Steiner 2-designs, for which this lower bound is attained.



\section{Banff designs}
As just noted, if ${\cal L}$ is the Levi graph of a $(v,k,\lambda)$ design $\cal D$, then $h({\cal L})\ge v$. We want to study designs for which equality holds.
\begin{definition}
A $(v,k,\lambda)$ design $\cal D$ whose Levi graph ${\cal L}$ has harmonious chromatic number equal to $v$ will be called a {\it Banff design}.
\end{definition}

Note that a $(v,3,1)$-design is a Steiner triple system of order $v$.
It is possible to show that 
the existence of a $(v,3,1)$ Banff design is equivalent to that of a {\it nesting} of a STS$(v)$ into a $(v,4,2)$ design. 
Recall that a $(v,3,1)$-design $(V,{\cal B})$ can be {\it nested} if there is a mapping $ \varsigma: {\cal B} \to V$ such that $(V, \{B \cup \varsigma(B)\}:
B \in {\cal B})$ is a $(v, 4, 2)$-design. A nested STS$(v)$ can be harmoniously colored with $v$ colors, by coloring each element of $V$ with a different color, and coloring each block of $B\in \cal B$ with the color given to the point $\varsigma(B)$. Conversely, any harmonious coloring of an STS$(v)$ using exactly $v$ colors gives rise to a nesting by  adding to the block $B$ the only point that has the same color of $B$.
The existence problem for nested Steiner triple systems  has been completely settled in \cite{LR,S}, and using the results in these papers we can state the following. 
\begin{theorem}[\cite{LR,S}] \label{nesting} 
There exists a $(v,3,1)$ Banff design if and only if $v\equiv1$ $($mod $6)$.
\end{theorem}

In a recent work \cite{BKS}, using a variety of direct  and recursive constructions, the authors prove the existence of a $(v,4,1)$ Banff design for {\em all} admissible values of $v$. 

\smallskip 

In what follows we will use difference methods to construct some classes of Banff $(v,k,\lambda)$ designs. We first introduce the notion of a {\em Banff difference family}, a difference family having extra properties that guarantee that the design arising from it is a Banff design, and then apply this notion in various constructions in the rest of the paper.

\begin{definition}
       Let $G$ be an additive (not necessarily commutative) group of order $v$. A $(v,k,\lambda)$ difference family in $G$ is a set
       ${\cal F}$ of $k$-subsets of $G$ (called {\it base blocks} of $\cal F$) such that the list
 $\Delta {\cal F} := \{x - y : x,y \in B, x \ne y , B\in{\cal F}\}$ contains every element of $G \setminus \{0\}$ exactly $\lambda$ times. The difference family is called {\em cyclic} if $G$ is cyclic.
\end{definition}

Note that if $\cal F$ has size $n$, then $\Delta {\cal F}$ has size $k(k-1)n$. Hence a trivial necessary condition for
the existence of a $(v,k,\lambda)$ difference family is that $\lambda(v-1)=k(k-1)n$.
When the base blocks are pairwise disjoint we speak of a {\it disjoint} difference family. When
$\cal F$  consists of a single base block $B$ we say that $B$ is a {\it difference set}.
The {\it development} of a $(v,k,\lambda)$ difference family ${\cal F}$ is the multiset $dev{\cal F}=\{B+g \ | \ g\in G; B\in {\cal F}\}$ of all
possible translates of its base blocks.
The pair $(G,dev{\cal F})$ is a $(v,k,\lambda)$ design admitting an automorphism group isomorphic to $G$ acting 
sharply transitively on the points.
The existence question for $(v,k,\lambda)$ difference families is in general quite hard, especially when $\lambda=1$.
It was solved for $k=3$ a long time ago by Peltesohn \cite{P}. Recently, and quite unexpectedly, it was also solved for $k=4$ in \cite{Zhang}.
%

If $B$ is a block of a difference family in an additive group $G$, by $-B$ we mean the set $\{-b \, | \, b\in B\}$.  Infinitely many classes of Banff designs can be obtained via {\it Banff difference families}, defined as follows.

\begin{definition}
A  {\it Banff difference family} is a disjoint difference family $ {\cal F} =\{B_1,\dots,B_n\}$
such that $0\notin B_i$ for every $i$, and $B_i \ \cap \ -B_j =\emptyset$ for every possible pair $(i,j)$.
\end{definition}

\begin{example}\label{simple}
A $(13,3,1)$ Banff difference family in $\mathbb{Z}_{13}$ 
is $\{\{7,8,11\},\{4,10,12\}\}$.
\end{example}

We note that
a necessary condition for the existence of a $(v,k,\lambda)$ Banff difference family is that $\lambda$ cannot exceed ${k-1\over2}$.
Indeed, by definition, the union of the $n$ base blocks and their {\it negatives} should be a set (not a multiset) of
size $2kn$ so that $2kn\leq v-1$. On the other hand, as already said, we also have $\lambda(v-1)=k(k-1)n$ and then $\lambda\leq{k-1\over2}$.

We point out that Banff difference families appear in the literature for Steiner triple systems (i.e. $(v,3,1)$ designs) with the name \emph{symmetric} difference families; the strong version of a conjecture of Nov\'ak states that any cyclic STS$(v)$ with $v\equiv 1 \pmod{6}$ is generated by a symmetric cyclic difference family (see \cite{No, FHW}).

Example \ref{simple} above 
falls in the class of radical difference families
which we will consider later. 
We hazard the conjecture that for any cyclic $(v,k,1)$ difference family ${\cal F}=\{B_1,\dots,B_n\}$ there is
a suitable $n$-tuple $(t_1,\dots,t_n)$ of elements of $\mathbb{Z}_v$ such that
$\{B_1+t_1,\dots,B_n+t_n\}$ is a Banff $(v,k,1)$ difference family. 
\begin{example}
Starting from the cyclic $(85,4,1)$ difference family ${\cal F}=\{B_1,\dots,B_7\}$ of the Handbook of Combinatorial Designs \cite{CD}
with base blocks
\small
$$B_1=\{0, 2, 41, 42\}, \quad B_2=\{0, 17, 32, 38\}, \quad B_3=\{0, 18, 27, 37\}, \quad B_4=\{0, 13, 29, 36\},$$
$$B_5=\{0, 11, 31, 35\},\quad B_6=\{0, 12, 26, 34\}, \quad B_7=\{0, 5, 30, 33\}$$
\normalsize
we have obtained, by computer search, a Banff $(85,4,1)$ difference family $\{B_1+t_1,\dots,B_7+t_7\}$ 
where $(t_1,\dots,t_7)=(2,7,10,20,19,60,58)$. Explicitly, the base blocks of this Banff difference family are
$$\{2, 4, 43, 44\}, \quad \{7, 24, 39, 45\}, \quad \{10, 28, 37, 47\}, \quad \{20, 33, 49, 56\},$$
$$\{19, 30, 50, 54\},\quad \{60, 72, 1 , 9\}, \quad \{58, 63, 3, 6\}.$$
\end{example}


\begin{proposition}
Every Banff difference family generates a Banff design.
\end{proposition}
\begin{proof}
Let ${\cal F}=\{B_1,\dots,B_n\}$ be a Banff $(v,k,\lambda)$ difference family in $G$ and let ${{\cal D}}=(G,dev{{\cal F}})$ be
the $(v,k,\lambda)$ design generated by ${\cal F}$. Consider the map $c: G \ \cup \ dev{{\cal F}} \rightarrow G$
defined by the rule
$$c(g)=g \quad\forall g\in G;\quad\quad c(B_i+t)=t\quad \mbox{for $1\leq i\leq n$ and for any $t\in G$}.$$
Let $\{g,X\}$ be an edge of the Levi graph ${\cal L}$ associated with ${\cal D}$.
We have $g\in X=B_i+t$ for a suitable pair $(i,t) \in \{1\dots,n\}\times G$ and then $g\neq t$, for otherwise we would have $0\in B_i$
contradicting one of the properties of a Banff difference family. This means that $c(g)\neq c(X)$ so that
$c$ is a proper vertex coloring of ${\cal L}$ using $v$ colors.
To prove the assertion we have to show that $c$ is harmonious.

Let $\{g,X\}$ and $\{h,Y\}$ be two edges of ${\cal L}$ receiving the same pair of colors so that we have either
\begin{equation}\label{case1}
c(g)=c(h)\quad\mbox{and}\quad c(X)=c(Y)
\end{equation}
or
\begin{equation}\label{case2}
c(g)=c(Y)\quad\mbox{and}\quad c(X)=c(h).
\end{equation}
We have $g\in X=B_i+t$ and $h\in Y=B_j+u$ for suitable
pairs $(i,t)$ and $(j,u)\in \{1,\dots,n\}\times G$. 
In case (\ref{case1}) we have $g=h$ and $t=u$ so that we have $g\in (B_i+t) \ \cap \ (B_j+t)$.
It follows that $g=b+t$ and $g=b'+t$ for suitable elements $b\in B_i$ and $b'\in B_j$.
This implies $b=b'$ which is possible only for $i=j$ since ${\cal F}$ is disjoint.
We conclude that $X=Y$ and then the two edges $\{g,X\}$ and $\{h,Y\}$ are the same.

In case (\ref{case2}) we have $g=u$ and $t=h$. It follows that $g=b+h$ and $h=b'+g$ 
for suitable elements $b\in B_i$ and $b'\in B_j$. This implies $b=-b'$, hence $B_i \ \cap \ -B_j$ is not empty 
contradicting the hypothesis that ${\cal F}$ is a Banff difference family.

We conclude that distinct edges of ${\cal L}$ always receive a distinct pair of colors, i.e., $c$ is harmonious.
\end{proof}

If $q$ is a prime power, a $(q,k,1)$ difference family in ${\cal F}_q$ is 
said to be {\it radical} \cite{radical,pairwise} if its base blocks are either cosets of the group of $k$-th roots of unity of ${\cal F}_q$ in the case $k$ odd,
or the union of a coset of the $(k-1)$-th roots of unity and zero in the case $k$ even.
\begin{proposition}
A radical $(q,k,1)$ difference family with $k$ odd is also a Banff difference family.
\end{proposition}
\begin{proof}
Let ${\cal F}$ be a radical $(q,k,1)$ difference family with $k$ odd.
It is obvious that $\cal F$ is disjoint and that 0 does not lie in any base block.
Assume that $B$ and $B'$ are base blocks such that $B \ \cap \ -B'$ is not empty. 
Then we would have an element $b\in B$ such that $-b\in B'$.
This implies that $B'=-B$ hence $\Delta B'=\Delta B$, and then $B=B'$ else $\cal F$ would have repeated differences.
\end{proof}

Using the necessary and sufficient conditions for the existence of a radical $(q,5,1)$ difference family
\cite{precious} we can state the following.

\begin{corollary}
There exists a Banff $(q,5,1)$ Banff design for any prime power $q=20n+1$ such that
${1+\sqrt{5}\over2}$ is not a $2^{e+1}$-th power of ${\cal F}_q$ where $2^e$ is the largest power of $2$ in $n$. 
\end{corollary}

\section{Exact colorings of the Levi graph}

A harmonious coloring in which each pair of colors appears in {\it exactly} one edge is called an {\it exact coloring} \cite{MR1477743}. 
We noted above that a Banff difference family may exist only if  $\lambda\leq{k-1\over2}$. Let us point out that the extremal case $\lambda={k-1\over2}$ 
deserves special attention since it gives rise to an exact coloring: indeed the Levi graph in this case has exactly ${v\choose2}$ edges. 

A large class of these extremal Banff difference families can be obtained from the main theorem in \cite{disjointBuratti}.
\begin{theorem}
Let $G$ be an abelian group of odd order $v$ and let $A$ be a group of odd order $k$ of fixed-point-free automorphisms of $G$.
Then there exists a Banff $(v,k,{k-1\over2})$ difference family in $G$.
\end{theorem}
\begin{proof}
Looking at the proof of Theorem 3.1 in \cite{disjointBuratti} one can see that there is a suitable 
${v-1\over2k}$-subset $Y$ of $G\setminus\{0\}$ such that $Y \ \cup \ -Y$
is a complete system of representatives for the $A$-orbits on $G\setminus\{0\}$ (note that there is a typo on the size of $Y$ in \cite{disjointBuratti}).
Then we can see that the set of all $A$-orbits represented by the elements of $Y$ is a Banff $(v,k,{k-1\over2})$ difference family in $G$.
\end{proof}

In view of the necessary and sufficient condition for the existence of at least one pair $(G,A)$ as in the statement 
of the above theorem (see Proposition 3.4 in \cite{disjointBuratti}), we can state the following.
\begin{corollary}
If $k\geq3$ is odd and the maximal prime power factors of $v$ are all congruent to $1$ $($mod $2k)$,
then there exists an exact coloring of the Levi graph of a $(v,k,{k-1\over2})$-design.
\end{corollary}

\section{Banff Steiner 2-designs}\label{steiner}

In this section we will consider only $(v,k,1)$ designs, that is, Steiner 2-designs. We start by considering 
projective planes.
 
The finite projective plane $PG(2,q)$, $q$ a prime power, can be seen as a $(q^2+q+1,q+1,1)$-design, where the blocks correspond to the lines of the plane; it is well known that this design is generated by a cyclic $(q^2+q+1,q+1,1)$-difference set (a Singer difference set, see for instance \cite{Stin}, Section 3.3). 

Note that the Levi graph of $PG(2,q)$ is a $(q+1,6)$-cage \cite{EJ}.
A preliminary investigation about the harmonious chromatic number of this graph 
was done in \cite{Christian} where it was established that this number is at most $q^2+q+2$.
Here we prove that it is always equal to $q^2+q+1$, that is, PG$(2,q)$ is a Banff design.

\begin{theorem}
Any desarguesian projective plane is a Banff design.
\end{theorem}
\begin{proof}
It is enough to prove that there exists a $(q^2+q+1,q+1,1)$ Banff difference set for any prime power $q$.
Let $B=\{b_0,b_1,\dots,b_q\}$ be a Singer $(q^2+q+1,q+1,1)$ difference set.
Consider the subset $X$ of $\mathbb{Z}_{q^2+q+1}$ defined by 
$$X=\left\{{b_i+b_j\over2} \ \Big| \ 0\leq i\leq j\leq q\right\}.$$ 
We have $|X|={(q+1)(q+2)\over2}<q^2+q+1$ so that we can take an element $t\in\mathbb{Z}_{q^2+q+1}\setminus X$.
Consider the set $B':=B-t$ which, of course, is still a $(q^2+q+1,q+1,1)$ difference set.

If $0\in B'$ we would have $t=b_i$ for some $i$. On the other hand, we have $b_i={b_i+b_i\over2}\in X$
contradicting the choice of $t$. 
Thus $B'$ does not contain $0$.

If $B' \ \cap \ -B'$ is not empty there would be a pair $(i,j)$ such that $b_i-t=-(b_j-t)$ and then $t={b_i+b_j\over2}\in X$
contradicting again the choice of $t$.

We conclude that $B'$ is a Banff difference set.
\end{proof}

In the following, given a prime power $q\equiv1$ (mod $e$), we denote by $C^e$ the subgroup of ${\cal F}_q^*$
of index $e$, that is the group of non-zero $e$-th powers of ${\cal F}_q$. The cosets of $C^e$ in ${\cal F}_q^*$ will be denoted 
by $C^e_0$, $C^e_1$, \dots, $C^e_{e-1}$.

Given a prime power $q\equiv1$ (mod $k(k-1))$, a $k$-subset $B=\{b_0,b_1,\dots,b_{k-1}\}$ of ${\cal F}_q$ will be said to be a $(q,k,1)$ Banff set if the
following conditions hold:
\begin{itemize}
\item[(1)] the list of differences $\vec\Delta B:=\{b_i-b_j \ | \ 0\leq i < j \leq k-1\}$ is a complete system of representatives for the cosets
of $C^{k(k-1)/2}$ in ${\cal F}_q^*$;
\item[(2)] $B$ is a partial system of representatives for
for the cosets of $C^{k(k-1)/2}$ in ${\cal F}_q^*$.
\end{itemize}

The following is an adaptation of the {\it lemma on evenly distributed differences} by Wilson \cite{W} (see also \cite{BJL}).

\begin{lemma}\label{Banff set}
If there exists a $(q,k,1)$ Banff set, then there exists a $(q,k,1)$ Banff difference family. 
\end{lemma}
\begin{proof}
Let $B=\{b_0,b_1,\dots,b_{k-1}\}$ be a $(q,k,1)$ Banff set and take $S$ to be a complete system of representatives for the cosets of $\{1,-1\}$ in $C^{k(k-1)/2}$.

For $s\in S$ denote by $B_s$ the set $\{b_0s, b_1s,\dots,b_{k-1}s\}$, and set ${\cal F}=\{B_s \ | \ s\in S\}$. Condition (1) guarantees that ${\cal F}$ is a
$(q,k,1)$ difference family (this is the mentioned lemma by Wilson).
Obviously, no block of ${\cal F}$ contains zero.
Now assume that we have $$B_s \ \cap \ B_t\neq\emptyset\quad {\rm or}\quad B_s \ \cap \ -B_t\neq\emptyset$$
for some elements $s, t\in S$. In this case we have $b_is=\varepsilon b_jt$ for suitable $i, j\in\{0,\dots,k-1\}$ and $\varepsilon\in\{1,-1\}$.
Thus we have $b_ib_j^{-1}=\varepsilon s^{-1}t\in C^{k(k-1)/2}$ so that, by (2), $b_i=b_j$ and then $s^{-1}t=\varepsilon\in\{1,-1\}$.
By definition of $S$, this is possible only for $s=t$ and $\varepsilon=1$. 
We conclude that condition (2) guarantees that ${\cal F}$ is a Banff difference family. 
\end{proof}

We need the following consequence of the theorem of Weil on multiplicative character sums
(see \cite{LN}, Theorem 5.41), that is Theorem 2.2 in \cite{BP}.
\begin{lemma}\label{BP} 
Let $q\equiv 1 \pmod{e}$ be a prime power, let $\{c_1,\dots,c_k\}$ be a $k$-subset  of ${\cal F}_q$, and let $(\phi_1,\dots,\phi_k)$ be an ordered $k$-tuple of $\mathbb{Z}_e^k$.
Then the set $X:=\{x\in {\cal F}_q: x-c_i\in C_{\phi_i}^e \,\,{\rm  for }\,\, i=1,\dots,k \}$
is not empty provided that $q$ is greater than a suitable bound $Q(e,k)$.
\end{lemma}

The bound $Q(e,k)$ is generally unwieldy; a more manageable approximation by excess is given by $k^2e^{2k}$. 


The following is an adaptation of some of the constructions in \cite{BP}.

\begin{theorem}
There exists a Banff $(q,k,1)$-design for any prime power $q\equiv1$ $($mod $k(k-1))$ sufficiently large.
\end{theorem}
\begin{proof}
Set $e={k(k-1)\over2}$, let $q=en+1$ be a prime power greater than $Q(e,k)$, and let
$I$ be the set of pairs $(i,j)$ with $0\leq i<j\leq k-1$.
Fix any bijection $\phi: I\longrightarrow \mathbb{Z}_e$ and construct a set $B=\{b_0,\dots,b_{k-1}\}\subset{\cal F}_q^*$ as follows.
Start taking $b_0$ arbitrarily in $C^e_0$ and then take the other elements 
$b_1$, \dots, $b_{k-1}$ iteratively, one by one, according to the rule that
once that $b_{j-1}$ has been chosen, we pick $b_j$ arbitrarily in the set 
$$X_j=\{x\in {\cal F}_q: x-0\in C^e_j \ {\rm and
} \ x-b_i\in C^e_{\phi(i,j)} \quad {\rm for} \ 0\leq i\leq j-1\}$$
which is not empty  by Lemma \ref{BP}.

It is quite evident that this set $B$ is a $(q,k,1)$ Banff set. The assertion then follows from Lemma \ref{Banff set}.
\end{proof}

Applying the above theorem with $k=4$ together with Lemma \ref{BP} we can say that there exists a
$(q,4,1)$ Banff set for any prime power $q=12n+1>Q(6,4)=9\,152\,353$. 
On the other hand, we have checked by computer that a Banff set exists in ${\cal F}_q$ for any prime power value of $q\equiv1$ (mod 12) up to that bound;
in particular, the Banff sets we constructed are, 
in all except for four small prime fields,
of the form $\{1,x,x^2,x^3\}$, $x \in {\cal F}_q$. We report in Table \ref{ta:banffset1} and \ref{ta:banffset2} the Banff sets for some small values of $q$.

\begin{center}
\begin{table}[ht]
\centering
\begin{tabular}{||c c || c c|| c c ||} 
 \hline
$q$  & Banff set & $q$ & Banff set & $q$ & Banff set \\ 
 \hline\hline
 13 & $\{1,2,4,10\}$ & 157 & $\{1,6, 6^2,6^3\}$ & 313 & $\{1,10, 10^2,10^3\}$\\ 
 \hline
 37 & $\{1,13, 13^2,13^3\}$ & 181 & $\{1,2,13,19\}$ & 337 & $\{1,65, 65^2,65^3\}$ \\
 \hline
 61 & $\{1,7,11,12\}$ & 193 & $\{1,70, 70^2,70^3\}$ & 349 & $\{1,18, 18^2,18^3\}$\\
 \hline
 73 & $\{1, 29, 29^2, 29^3\}$ & 229 & $\{1,18, 18^2,18^3\}$ & 373 & $\{1, 32, 32^2,32^3\}$ \\
 \hline
 97 & $\{1, 41, 41^2, 41^3\}$ & 241 & $\{1,66, 66^2, 66^3\}$ & 397 & $\{1,13, 13^2,13^3\}$\\  
 \hline
 109 & $\{1,2,11,13\}$ & 277 & $\{1,72, 72^2,72^3\}$ & 409 & $\{1,33, 33^2,33^3\}$\\  
 \hline
\end{tabular}
\caption{Banff sets in ${\cal F}_q$, $q$ prime, $q\le 409$.\label{ta:banffset1}}
\end{table}
\end{center}

\begin{center}
\begin{table}[ht]
\centering
\begin{tabular}{||c c c|| c c c||} 
 \hline
$q$  & $f(z)$ & $x$ & $q$ & $f(z)$ & $x$ \\ 
 \hline\hline
 25 & $z^2+z+2$ & $z^7$ & 361 & $z^2+z+2$ &$z^{11}$\\ 
 \hline
 49 & $z^2+z+3$ & $z^{11}$ & 449 & $z^2+z+7$ & $z^{83}$ \\
 \hline
 121 & $z^2+z+7$ & $z^{17}$ & 841 & $z^2+z+3$ & $z^{71}$\\
 \hline
 169 & $z^2+z+2$ & $z^{23}$ & 961 & $z^2+2z+3$ & $z^{49}$ \\
 \hline
 289 & $z^2+z+3$ & $z^{91}$ &  &  & \\  
 \hline
\end{tabular}
\caption{$x$ such that $\{1,x,x^2,x^3\}$ is a Banff set in ${\cal F}_q\simeq\mathbb{Z}_p[z]/(f(z))$, $q=p^2\le 961$.\label{ta:banffset2}}
\end{table}
\end{center}
\comment{}

\begin{corollary}
There exists a Banff $(q,4,1)$-design for any prime power $q\equiv1$ $($mod $12).$
\end{corollary}

However note that already in the case $k=5$ the situation changes; we can assert that a $(q,5,1)$ Banff set for any prime power $q=20n+1>Q(10,5)=122\,500\,800\,001$, so that checking all the values of $q$ below the bound becomes unmanageable.

\comment{
In the case $k=4$ let us check, for instance, that 
$B=\{1,13, 13^2,13^3\}=\{1, 13, 21, 14\}$
is a $(37,4,1)$ Banff set.
The cosets of $C^6$ in $\mathbb{Z}_{37}^*$ are listed below 
\small
$$C^6_0=\{1,10,11,26,27,36\},\quad C^6_1=\{2, 20, 22, 15,17, 35\},\quad C^6_2=\{4, 3, 7, 30,34, 33\},$$
$$C^6_3=\{8, 6, 14, 23,31, 29\},\quad C^6_4=\{16, 12, 28, 9,25, 21\},\quad C^6_5=\{32, 24, 19, 18,13, 5\}.$$
\normalsize
It is readily seen that $\vec\Delta B=\{1,2,3,23,21,24\}$ and that we have
$$1\in C^6_0, \quad 2\in C^6_1,\quad 3\in C^6_2,\quad 23\in C^6_3,\quad 21\in C^6_4,\quad 24\in C^6_5$$
so that $\vec\Delta B$ is a complete system of representatives for the cosets
of $C^6$ in $\mathbb{Z}_{37}^*$.
We also see that the four elements 1, 2, 4, 25 of $B$ belong to distinct cosets of $C^6$ in $\mathbb{Z}_{37}^*$.

The above guarantees that $B$ is a $(37,4,1)$ Banff set. A system of representatives for the cosets of $\{1,-1\}$ in $C^6$ 
is $\{1,10,11\}$. Thus, applying Lemma \ref{Banff set}, we
get a Banff $(37,4,1)$ difference family whose blocks are $B$, $10B=\{10,20,3,28\}$, and $11B=\{11, 22, 7, 16\}$.
}

In the case $k=4$ let us check, for instance, that 
$B=\{1,13, 13^2,13^3\}=\{1, 13, 21, 14\}$
is a $(37,4,1)$ Banff set.
The cosets of $C^6$ in $\mathbb{Z}_{37}^*$ are listed below 
\small
$$C^6_0=\{1,10,11,26,27,36\},\quad C^6_1=\{2, 20, 22, 15,17, 35\},\quad C^6_2=\{4, 3, 7, 30,34, 33\},$$
$$C^6_3=\{8, 6, 14, 23,31, 29\},\quad C^6_4=\{16, 12, 28, 9,25, 21\},\quad C^6_5=\{32, 24, 19, 18,13, 5\}.$$
\normalsize
It is readily seen that $\vec\Delta B=\{1,7,8,12,13,20\}$ and that we have
$$1\in C^6_0, \quad 20\in C^6_1,\quad 7\in C^6_2,\quad 8\in C^6_3,\quad 12\in C^6_4,\quad 13\in C^6_5$$
so that $\vec\Delta B$ is a complete system of representatives for the cosets
of $C^6$ in $\mathbb{Z}_{37}^*$.
We also see that the four elements 1, 13, 14, 21 of $B$ belong to distinct cosets of $C^6$ in $\mathbb{Z}_{37}^*$.

The above guarantees that $B$ is a $(37,4,1)$ Banff set. A system of representatives for the cosets of $\{1,-1\}$ in $C^6$ 
is $\{1,10,11\}$. Thus, applying Lemma \ref{Banff set}, we
get a Banff $(37,4,1)$ difference family whose blocks are $B$, $10B=\{10,19,25,29\}$, and $11B=\{11, 32, 9, 6\}$.

\section{Final remarks}

Let us remark that there are values of the parameters $(v,k,\lambda)$ for which designs exist but no Banff design can exist. For instance Theorem \ref{nesting} 
implies that there is no Banff $(v,3,1)$ design with $v\equiv3$ (mod 6). The following result shows that $h({\cal L})>v$ if $r+(v+1)/2 >v$, that is, if $r>(v-1)/2$; for instance, it tells us that the harmonious chromatic number of a $(16,4,2)$-design must be at least 19. 
\begin{theorem}\label{christian}
The Levi graph of a $2$-design with $v$ points and replication number $r$ has harmonious chromatic number at least equal to $r+\frac{v+1}{2}$.
\end{theorem}
\begin{proof}
Let ${\cal D}=(V,{\cal B})$ be $(v, k,\lambda)$-design with point set $V= \{1,\dots,v\}$ and let $\cal L$ be the Levi graph of $\cal D$.
Assume that $\varsigma$ is a harmonious $h$-coloring of $\cal L$ with color set $C=\{1,\dots,h\}$.
As already said, we necessarily have $h\geq v$ since the colors of any pair of distinct points have to be distinct. 
Thus the set of points $V$ is a subset of the set $C$ of colors.
Without loss of generality we may assume that $\varsigma(i)=i$ for $1\leq i\leq v$.
For $1\leq i\leq h$,  let ${\cal B}_i$ be the set of blocks colored $i$.  Obviously, the blocks of ${\cal B}_i$ are pairwise disjoint so that
we have \begin{equation}\label{disjoint}|{\cal B}_i|\leq \biggl{\lfloor}{v\over k}\biggl{\rfloor}\quad \forall i\in C.\end{equation}
It is also clear that $\sum_{i=1}^h |{\cal B}_i| =b$. 
Let ${\cal B}^*:=\bigcup_{i=1}^{v}{\cal B}_i$ be the set of all blocks whose colors belong to $V$ and set $|{\cal B}^*|=b^*$.
For $i=1,\dots,v$, let $r_i$ be the number of blocks of ${\cal B}^*$ containing the point $i$.
Consider the set $S$ of all flags $(i,B)$ of the design $\cal D$ having $B$ in ${\cal B}^*$. For each $\overline{B}\in {\cal B}*$, the number of pairs $(\overline{i},B)$ belonging to 
$S$ with $\overline i$ fixed is $r_{\overline{i}}$ so that we have $|S|=\sum_{i=1}^vr_i$. The number of pairs $(i,\overline{B})$ belonging to 
$S$ with $\overline B$ fixed is clearly equal to $k$ so that we have $|S|=k\cdot b^*$.
Comparing the obtained equalities we get 
\begin{equation}\label{b*} \sum_{i=1}^v{r_i\over k}=b^*.\end{equation}
Now, given $i\in V$, let $C_i$ be the set of colors of the $r_i$ blocks of ${\cal B}^*$ passing through the point $i$.
A block $B$ of ${\cal B}_i$ has empty intersection with the set $C_i \ \cup \ \{i\}$. Indeed if $c\in C_i \ \cap \ B$,
then there is a flag $(i,B')$ of $\cal D$ with $B'\in{\cal B}_c$ since $c\in C_i$ and, at the same time, $(c,B)$ is a flag of $\cal D$ since 
$c\in B$. It follows that the two flags $(i,B')$ and $(c,B)$ receive the same pair of colors $\{i,c\}$ against the definition of harmonious coloring.
Also, we cannot have $i\in B$ otherwise $i$ and $B$ would be adjacent vertices of $\cal L$ with the same color.
Thus the union of the blocks of ${\cal B}_i$, which is a set of size $k\cdot|{\cal B}_i|$, is contained in $V\setminus(C_i \ \cup \ \{i\})$ which has size $v-1-r_i$.
It follows that $|{\cal B}_i|\leq {v-1-r_i\over k}$ and then, taking into account of (\ref{b*}), $$b^*=\sum_{i=1}^v|{\cal B}_i|\leq \sum_{i=1}^v {v-1-r_i\over k}=\sum_{i=1}^v {v-1\over k}-\sum_{i=1}^v {r_i\over k}
={v(v-1)\over k}-b^*.$$
We deduce that $b^*\leq{v(v-1)\over 2k}$. At this point, also taking into account (\ref{disjoint}), we can write
$${vr\over k}=b=\sum_{i=1}^h|{\cal B}_i|=\sum_{i=1}^v|{\cal B}_i|+\sum_{i=v+1}^h|{\cal B}_i|\leq {v(v-1)\over 2k}+(h-v)\left\lfloor{v\over k}\right\rfloor.$$
Setting ${v\over k}=s$, the above inequality can be rewritten as
$$rs\leq {v-1\over 2}s+(h-v)\lfloor s\rfloor.$$
Solving this with respect to $h$ yields $h\geq (r-{v-1\over 2}){s\over \lfloor s\rfloor}+v$. Obviously, we have  ${s\over \lfloor s\rfloor}\ge1$
and then $h\geq (r-{v-1\over 2})+v=r+{v+1\over2}$.
\end{proof}

Note that the harmonious chromatic number of the Levi graph may be strictly greater than the bounds we presented here. 
For instance, consider the  Levi graph $\cal L$ of the point-plane design $\cal D$ associated with PG$(3,2)$, that is a $(15,7,3)$-design: 
both Theorem \ref{christian}, and the fact that the design has 15 points, imply that $h({\cal L})\ge 15$.
Yet, 
 we have established that $h({\cal L})=20$ by means of a computer program. 
The design $\cal D$ can be viewed as the pair $(\mathbb{Z}_{15},dev B)$ where
$B=\{0,1,2,4,5,8,10\}$. A harmonious chromatic 20-coloring $\varsigma$ of $\cal L$ with set of colors $C=\mathbb{Z}_{15} \ \cup \ \{c_1,c_2,c_3,c_4,c_5\}$ 
can be obtained by taking
$\varsigma(i)=i$ for every $i\in \mathbb{Z}_{15}$ and $\varsigma(B+i)=\sigma_i$ where the $\sigma_i$s are defined as follows: 
$$(\sigma_0,\sigma_1,\dots,\sigma_{14})=(3, 7, 13, 14, 1, 8, c_1, 4, 11, c_2, 6, c_3, 0, c_4, c_5).$$

\section*{Acknowlegements}
We would like to thank Banff International Research Station for Mathematical Innovation and Discovery and the organizers of the BIRS workshop 
``Extremal Graphs arising from Designs and Configurations''.

We also thank Doug Stinson for interesting discussions exploring the connection between harmonious colorings and nestings.

Finally, we thank the referees for their  
careful reading of our paper, and for their comments, corrections and suggestions that improved this work considerably.

\end{document}